\documentclass[12pt]{amsart}
\usepackage{amsfonts, amssymb, amsmath, amsthm}
 \usepackage{graphicx,latexsym}
 \usepackage[hmargin=1.2in, vmargin=1.2in]{geometry}

\newtheorem{theorem}{Theorem}
\newtheorem{proposition}{Proposition}
\newtheorem{lemma}{Lemma}
\newtheorem{corollary}{Corollary}

\theoremstyle{definition}

\newtheorem{example}{Example}

\theoremstyle{remark}
\newtheorem{remark}{Remark}

\begin{document}

\title[Gabor frames and asymptotic behavior of distributions]{Gabor frames and asymptotic behavior of Schwartz distributions}

\author[S. Kostadinova]{Sanja Kostadinova}
\address{Faculty of Electrical Engineering and Information Technologies\\ Ss. Cyril and
Methodius University\\ Rugjer Boshkovik bb\\ 1000 Skopje, Macedonia} \email{ksanja@feit.ukim.edu.mk}

\author[K. Saneva]{Katerina Saneva}
\address{Faculty of Electrical Engineering and Information Technologies\\ Ss. Cyril and
Methodius University\\ Rugjer Boshkovik bb\\ 1000 Skopje, Macedonia} \email{saneva@feit.ukim.edu.mk}

\author[J. Vindas]{Jasson Vindas}
\thanks{S. Kostadinova and K. Saneva gratefully acknowledge support by Ministry of Education and Science of the Republic of Macedonia, through the grant 16-2087/1.}
\thanks{J. Vindas gratefully acknowledges support by Ghent University, through the BOF-grant 01N01014.}

\address{Department of Mathematics\\ Ghent University\\ Krijgslaan 281 Gebouw
S22\\ 9000 Gent, Belgium}
\email{jvindas@cage.UGent.be}

\begin{abstract}We obtain characterizations of asymptotic properties of Schwartz distribution by using Gabor frames. Our characterizations are indeed Tauberian theorems for shift asymptotics (S-asymptotics) in terms of short-time Fourier transforms with respect to windows generating Gabor frames. For it, we show that the Gabor coefficient operator provides (topological) isomorphisms of the spaces of tempered distributions $\mathcal{S}'(\mathbb{R}^d)$ and distributions of exponential type $\mathcal{K}'_{1}(\mathbb{R}^{d})$ onto their images.

\end{abstract}
\maketitle

\section{Introduction}

Asymptotic behavior is an important notion in distribution theory. The subject has been studied by many authors and applications have been developed in diverse areas such as mathematical physics, number theory, and differential equations; see, e.g.,  the monographs \cite{Estrada,m-l,PSV,VDZ}, references therein, and the recent article \cite{y-e}. The theory of asymptotic behavior of generalized functions has also shown to be quite useful in Tauberian theory for several integral transforms \cite{p-sWiener,prof44,VDZ}. In recent years, characterizations of asymptotic properties of distributions via wavelet analysis have been extensively investigated \cite{KPSVRidgelet,KV,PT,Saneva1,S,Vindas22}.

The purpose of this article is to study the so-called S-asymptotic behavior of a distribution through Gabor frames. Note that
Gabor frames have already been used in other works as an effective tool in the description local and microlocal properties of Schwartz distributions \cite{JPTT,RW,SAK}. Our main result is a characterization of the S-asymptotic behavior in terms of the short-time Fourier transform. The authors and Pilipovi\'{c} have recently obtained various Tauberian theorems for short-time Fourier transforms of distributions \cite{KPSV}. We show here in Section \ref{Section Asymptotic behavior of distributions} that those Tauberian theorems
can be considerably improved by discretizing the frequency variable if one employs windows generating a Gabor frame. We also derive a Tauberian theorem for non-decreasing functions and study connections with Wiener-type kernels.

 An important technical tool in the proofs of our Tauberian theorems from Section \ref{Section Asymptotic behavior of distributions} is a characterization of bounded sets of distributions of exponential type and tempered distributions in terms of growth estimates for Gabor frame coefficients, which will be obtained in Section \ref{owet}. It should be mentioned that the convergence of the Gabor frame series of a tempered distribution is a known fact in time-frequency analysis \cite[Chap.~12]{gr01}; however, the Tauberian problem considered in this article requires to establish that the Gabor coefficient operator indeed provides \emph{topological} isomorphisms of the spaces of distributions of exponential type and tempered distributions onto their images. We therefore revisit the connection between Gabor frames and distributions in Section \ref{owet}.

\section{Preliminaries}
\label{preli}
\subsection{Notation}
The translation and modulation operators are denoted as $T_{x}f(\: \cdot \:)=f( \: \cdot \:-x)$ and $M_{\xi}f(\: \cdot \:)=e^{2\pi i \xi
\:\cdot\:}f(\:\cdot\:),$ $ x,\xi\in\mathbb{R}^d. $ The operators  $M_{\xi} T_x$ and $T_x M_{\xi}$ are called time-frequency shifts and we have $M_{\xi} T_x=e^{2\pi i x\cdot\xi}T_x M_{\xi}.$ The notation $\langle f,\varphi\rangle$ means dual pairing between a distribution $f$ and a test function $\varphi$, so that  $(f,\varphi)_{L^{2}}=\langle f,\overline{\varphi}\rangle$ if $f,\varphi\in L^{2}$. All dual spaces in this article are equipped with the strong dual topology. We fix the constants in the Fourier transform as $\widehat {\varphi}(\xi)=\int_{\mathbb{R}^{d}} e^{-2\pi i x\cdot
\xi}\varphi(x)dx.$


\subsection{{Spaces}}

\noindent Besides the standard Schwartz spaces of rapidly decreasing smooth test functions ${\mathcal S}(\mathbb{R}^d)$ and tempered distributions ${\mathcal S}'(\mathbb {R}^d)$, we will also work with the Hasumi-Silva test function space  $\mathcal{K}_{1}(\mathbb{R}^d)$ and its dual  \cite{hasumi,hoskins-pinto}.
The space ${\mathcal K}_1(\mathbb{R}^d )$ consists of exponentially rapidly decreasing smooth functions, that is,  $\varphi\in{\mathcal K}_1(\mathbb{R}^d )$ if $\varphi\in C^{\infty}(\mathbb{R}^{d})$ and
\begin{equation}
 \label{eqseminorm1}
 \|\varphi\|_{p}:=\sup_{x\in {\mathbb{R}^d }, \ |j|\leq p}e^{p|x|}|\varphi^{(j)}(x)| <\infty,\ \ \forall p\in{\mathbb N}_0.
 \end{equation} The Fr\'{e}chet space topology of ${\mathcal K}_1(\mathbb{R}^d)$ is generated by the family of norms (\ref{eqseminorm1}).  Its dual space  ${\mathcal K}'_1(\mathbb{R}^{d})$ consists of all distributions $f$ of exponential type, i.e., those of the form $f=\sum_{|j|\leq l}(e^{s |\:
\cdot\: |}f_j)^{(j)}$, where $f_j\in L^{\infty}(\mathbb{R}^d)$. Note that if $\varphi\in \mathcal{K}_{1}(\mathbb{R}^d)$, then its Fourier transform $\widehat{\varphi}$ extends to an entire function on $\mathbb{C}^d$.

We shall also need two Fr\'{e}chet spaces of double (two-sided) rapidly decreasing sequences, namely,
$$
\mathcal{S}(\mathbb{Z}^{2d})=\{ \{c_{k,n}\}\in \mathbb{C}^{\mathbb{Z}^{d}\times\mathbb{Z}^{d}}:\:\sup_{(k,n)\in\mathbb{Z}^{2d}}\left|c_{k,n}\right|(1+|k|+|n|)^{p}<\infty\ , \forall p\in\mathbb{N}_{0}\}
$$
and

$$
\mathcal{S}_{\exp,\text{pol}}(\mathbb{Z}^{2d})=\{ \{c_{k,n}\}\in \mathbb{C}^{\mathbb{Z}^{d}\times\mathbb{Z}^{d}}:\:\sup_{(k,n)\in\mathbb{Z}^{2d}}\left|c_{k,n}\right|e^{p|k|}(1+|n|)^{p}<\infty\ , \forall p\in\mathbb{N}_{0}\}.
$$
Clearly, they are (FS)-spaces (Fr\'{e}chet-Schwartz spaces) and their duals are Silva spaces ((DFS)-spaces),
$$
\mathcal{S}'(\mathbb{Z}^{2d})=\{ \{c_{k,n}\}\in \mathbb{C}^{\mathbb{Z}^{d}\times\mathbb{Z}^{d}}:\:\sup_{(k,n)\in\mathbb{Z}^{2d}}\left|c_{k,n}\right|(1+|k|+|n|)^{-p}<\infty\ , \mbox{ for some }  p\in\mathbb{N}_{0}\}
$$
and
$$
\mathcal{S}'_{\exp,\text{pol}}(\mathbb{Z}^{2d})=\{ \{c_{k,n}\}\in \mathbb{C}^{\mathbb{Z}^{d}\times\mathbb{Z}^{d}}:\:\sup_{(k,n)\in\mathbb{Z}^{2d}}\left|c_{k,n}\right|e^{-p|k|}(1+|n|)^{-p}<\infty\ , \mbox{ for some }  p\in\mathbb{N}_{0}\}.
$$


\subsection{Short-time Fourier transform}\label{STFT} The \textit{short-time Fourier transform} (STFT)  of a function $f$ with respect to a
window function $\psi$ is defined as
 \begin{equation}
\label{STFT}
V_{\psi} f(x,\xi )=\langle f, \overline{M_{\xi } T_{x} \psi}\rangle =\int _{\mathbb{R}^{d}}
f(t)\overline{\psi(t-x)}e^{-2\pi i\xi \cdot t} \ dt,\quad x,\xi \in {\mathbb{R}^d}.
\end{equation}
The integral in (\ref{STFT}) is well defined if $f\psi \in L^{1}(\mathbb{R}^d)$, while the dual pairing definition applies whenever $f\in \mathcal{A}'(\mathbb{R}^d)$ and $\psi\in \mathcal{A}(\mathbb{R}^d)$, where $\mathcal{A}(\mathbb{R}^d)$ is a time-frequency shift invariant topological vector spaces of functions. Furthermore, if the action of the time-frequency shifts is continuous on $\mathcal{A}(\mathbb{R}^d)$,  then $V_{\psi}f\in C(\mathbb{R}^{2d})$. See \cite{gr01} for the analysis of the STFT for the cases $\mathcal{A}=L^{2}$ and $\mathcal{S}$;  we refer to \cite{KPSV} for a STFT theory based on the spaces $\mathcal{K}'_{1}$ and  $\mathcal{K}_{1}$. We also mention the interesting paper \cite{b-o}, where window functions are even allowed to be distributions.

We shall need some mapping properties of the STFT.  The authors have shown in \cite[Prop.~3.1 \& Lemma~3.2]{KPSV} the following results:
\begin{equation}
\label{boundSTFTK} f,\psi \in \mathcal{K}_{1}(\mathbb{R}^d) \  \Rightarrow (\forall p) \ \  |V_{\psi}f(x,\xi)|\leq C_{p}e^{-p|x|}(1+|\xi|)^{-p}, \
\forall x, \xi\in\mathbb {R}^{d}, \ \mbox{ and}
\end{equation}
\begin{equation}
\label{boundSTFTK'} f\in\mathcal{K}_{1}' (\mathbb{R}^d) \mbox{ and } \psi \in\mathcal{K}_{1}(\mathbb{R}^d) \ \Rightarrow (\exists p) \ \   |V_{\psi}f(x,\xi)|\leq C_{p}e^{p|x|}(1+|\xi|)^{p},\   \forall x, \xi\in\mathbb {R}^{d}.
\end{equation}
For the Schwartz spaces one has \cite[p. 228]{gr01}:
\begin{equation}
\label{boundSTFTS} f,\psi \in \mathcal{S}(\mathbb{R}^d) \  \Rightarrow (\forall p) \ \  |V_{\psi}f(x,\xi)|\leq C_{p}(1+|x|+|\xi|)^{-p}, \
\forall x, \xi\in\mathbb {R}^d, \ \mbox{ and }
\end{equation}
\begin{equation}
\label{boundSTFTS'} f\in\mathcal{S}' (\mathbb{R}^d) \mbox{ and } \psi \in\mathcal{S}(\mathbb{R}^d) \ \Rightarrow (\exists p) \ \ |V_{\psi}f(x,\xi)|\leq C_{p}(1+|x|+|\xi|)^{p},\   \forall x, \xi\in\mathbb {R}^d.
\end{equation}
It is important to point out that all these estimates can be shown to hold uniformly when $f$ runs over a bounded set of the corresponding space (see the proof of Lemma~\ref{lemma coeff op} below).
\subsection{Gabor frames}\label{Gabor frame} We briefly discuss in this subsection some notions from the theory of Gabor frames, see \cite[Chaps.~5--8]{gr01} for a complete account on the subject. Given a non-zero window function $\psi\in L^{2}(\mathbb{R}^d)$ and lattice parameters $\alpha, \beta>0$, the set of
time-frequency shifts \begin{equation}\label{GF} G(\psi, \alpha, \beta)=\{M_{\beta n}T_{\alpha k}\psi: k, n\in\mathbb{Z}^{d}\} \end{equation} is called a Gabor
frame for $L^2(\mathbb{R}^{d})$ if there exist $A, B>0$ (frame bounds) such that
 \begin{equation}\label{frame1}
A\|f\|^2_{L^2(\mathbb{R}^{d})}\leq \sum_{k,n\in \mathbb{Z}^{d}}|V_{\psi}f({\alpha k},
{\beta n})|^2\leq B\|f\|^2_{L^2(\mathbb{R}^{d})}, \quad f\in L^2(\mathbb{R}^{d}).
 \end{equation}
The Gabor frame
 operator
\begin{equation*}
Sf=S_{\psi,\psi}f=\sum_{k, n\in\mathbb{Z}^{d}}V_{\psi}f({\alpha k},
{\beta n})M_{\beta n}T_{\alpha k} \psi
\end{equation*}
is then bounded, positive, and invertible on $L^{2}(\mathbb{R}^{d})$. The (canonical) dual frame of $G(\psi, \alpha, \beta)$ is the Gabor frame $G(\gamma, \alpha, \beta)$, where the \emph{canonical} dual window is given by
$\gamma=S^{-1}\psi\in L^{2}(\mathbb{R}^{d})$. Every $f\in L^{2}(\mathbb{R}^{d})$ then possesses  the Gabor frame series expansions

\begin{equation}
\label{GFf}
 f=\sum_{k,n\in\mathbb{Z}^{d}}V_{\psi}f({\alpha k},
{\beta n})M_{\beta n}T_{\alpha k} \gamma= \sum_{k,n\in\mathbb{Z}^{d}}V_{\gamma}f({\alpha k},
{\beta n})M_{\beta n}T_{\alpha k}\psi
\end{equation}
with unconditional convergence in $L^{2}(\mathbb{R}^{d})$. The expression (\ref{GFf}) provides an explicit reconstruction of $f$ from the samples of its STFT on the separable lattice $\alpha \mathbb{Z}^{d}\times \beta \mathbb{Z}^{d}$.

One can show that if $G(\psi,\alpha,\beta)$ is a frame for $L^{2}(\mathbb R^{d})$ then $\alpha\beta\leq1$, but, in general, the condition $\alpha\beta\leq1$ does not guarantee that $G(\psi,\alpha,\beta)$ is a frame, see  \cite[Chap.~7]{gr01} for a discussion on this problem. Furthermore, the Balian-Low theorem \cite[Thm.~8.4.1, p.~163]{gr01} implies that if the window $\psi\in\mathcal{S}(\mathbb{R}^{d})$ (in particular if $\psi\in\mathcal{K}_{1}(\mathbb{R}^{d})$), then the stronger condition $\alpha \beta<1$ on the lattice parameters is necessary for $G(\psi,\alpha,\beta)$ to be a Gabor frame. The special case when $\psi(x)=e^{-\pi x^2}$ is the Gaussian is completely understood: $G(\psi,\alpha,\beta)$ is a Gabor frame for $L^{2}(\mathbb{R})$ if and only if $\alpha\beta<1$.

\section{Topological characterization of $\mathcal{K}'_1(\mathbb{R}^{d})$ and $\mathcal{S}'(\mathbb{R}^{d})$ via Gabor frame coefficients}
\label{owet}
In this section we show that topological concepts (such as boundedness and convergence) on $\mathcal{K}'_1(\mathbb{R}^{d})$ and $\mathcal{S}'(\mathbb{R}^{d})$ can be characterized through growth estimates of Gabor frame coefficients. For it, we need to study some properties of the so-called Gabor frame coefficient and synthesis operators. Let $\psi$ be a window (not necessarily generating a Gabor frame, unless explicitly stated). If $\psi\in \mathcal{K}_{1}(\mathbb{R}^{d})$
($\psi\in \mathcal{S}(\mathbb{R}^{d})$, resp.), we set, for ease of writing,
\begin{equation}
\label{Gabor coefficients}
c_{k,n}^{\psi}(f):=V_{\psi}f(\alpha k, \beta n)= \left\langle f,\overline {M_{\beta
n}T_{\alpha k} \psi}\right\rangle,
 \end{equation}
where $f\in
\mathcal K'_1(\mathbb{R}^{d})$ ($\mathcal{S}'(\mathbb{R}^{d})$, resp.). The Gabor coefficient and Gabor synthesis operators are formally defined as
\begin{equation}
\label{Gabor coefficient op}
C_{\psi}(f)=\{c_{k,n}^{\psi}(f)\}_{(k,n)\in\mathbb{Z}^{2d}}
 \end{equation}
 and
\begin{equation}
\label{Gabor synthesis op}
D_{\psi}(\{c_{k,n}\})=\sum_{(k,n)\in \mathbb{Z}^{2d}} c_{k,n} M_{\beta n}T_{\alpha k} \psi,
 \end{equation}
respectively. We begin by establishing the continuity of the Gabor coefficient operator on the spaces that we are concerned with.

\begin{lemma}\label{lemma coeff op} The Gabor coefficient operator $(\ref{Gabor coefficient op})$ is continuous in the following four cases:
\begin{equation}
\label{coeff op mappings 1}
C_{\psi}: \mathcal{K}_{1}(\mathbb{R}^{d})\to \mathcal{S}_{\exp,\textnormal{pol}}(\mathbb{Z}^{2d}), \quad  \ C_{\psi}: \mathcal{K}_1'(\mathbb{R}^{d})\to \mathcal{S}'_{\exp,\textnormal{pol}}(\mathbb{Z}^{2d}) \quad \quad (\psi\in \mathcal{K}_{1}(\mathbb{R}^{d})),
\end{equation}
\begin{equation}
\label{coeff op mappings 2}
C_{\psi}: \mathcal{S}(\mathbb{R}^{d})\to \mathcal{S}(\mathbb{Z}^{2d}), \quad \mbox{ and } \ C_{\psi}: \mathcal{S}'(\mathbb{R}^{d})\to \mathcal{S}'(\mathbb{Z}^{2d}) \quad \quad  (\psi\in \mathcal{S}(\mathbb{R}^{d})).
\end{equation}
\end{lemma}
\begin{proof}
 Since the four spaces $\mathcal{K}_{1}(\mathbb{R}^{d})$, $\mathcal{S}(\mathbb{R}^{d})$, $\mathcal{K}'_1(\mathbb{R}^{d})$, and $\mathcal{S}'(\mathbb{R}^{d})$ are bornological, the continuity of these mappings is an immediate consequence of uniform versions of the estimates (\ref{boundSTFTK}), (\ref{boundSTFTK'}), (\ref{boundSTFTS}),  and (\ref{boundSTFTS'}) over bounded sets of the corresponding spaces. Such uniform versions are explicitly proved in \cite[Prop. 3.1 \& Lemma 3.2]{KPSV} for the spaces $\mathcal{K}_{1}(\mathbb{R}^{d})$ and $\mathcal{K}'_1(\mathbb{R}^{d})$. The proofs for the cases $\mathcal{S}(\mathbb{R}^{d})$ and $\mathcal{S}'(\mathbb{R}^{d})$ are straightforward variants, but we include them for the sake of completeness. Consider the family of norms
\[\|\varphi\|'_{p}:= \sup_{x\in {\mathbb{R}^d }, \ |j|\leq p}|\varphi^{(j)}(x)| (1+|x|)^{p}, \quad \varphi\in\mathcal{S}(\mathbb{R}^{d}).
\]
Fix $p\in\mathbb{N}_{0}$ even. Then, for all $x,\xi\in\mathbb{R}^{d}$,
\begin{align*}
(1+|x|+|\xi|)^{p}|V_{\psi}\varphi(x,\xi)| 
&\leq C_{p} (1+|x|)^{p}\int_{\mathbb{R}^{d}}\left| (1-\Delta_{t})^{p/2}(\varphi(t)\overline{\psi}(t-x)) \right|dt
\\
&
\leq \tilde{C}_{p}\sum_{|j_{1}|+|j_{2}|\leq p}(1+|x|)^{p} \int_{\mathbb R^d} \left|\varphi^{(j_{1})}(t)\psi^{(j_2)}(t-x)\right| dt
\\
&
\leq 
\tilde{C}_{p}\|\varphi\|'_{p} \sum_{|j_{1}|+|j_{2}|\leq p} \int_{\mathbb R^d} (1+|u|)^{p}\left|\psi^{(j_2)}(u)\right| du
\\
&
\leq 
C'_{p} \|\varphi\|'_{p}\|\psi\|'_{p+n+1},
\end{align*}
which proves the assertion that (\ref{boundSTFTS}) holds uniformly on bounded subsets of $\mathcal{S}(\mathbb{R}^{d})$. Taking $x=\alpha k$ and $\xi=\beta n$, we obtain that $C_{\psi}:\mathcal{S}(\mathbb{R}^{d})\to \mathcal{S}(\mathbb{Z}^{2d})$ is a bounded mapping. Next, let $\mathfrak{B}\subset \mathcal{S}'(\mathbb{R}^{d})$ be bounded. This means that $\mathfrak{B}$ is equicontinuous, that is, there are $p\in\mathbb{N}_{0}$ and $C>0$ such that
\[
\sup_{f\in \mathfrak{B}} |\langle f,\varphi \rangle| \leq C \|\varphi\|'_{p}, \quad \mbox{for all } \varphi\in\mathcal{S}(\mathbb{R}^{d}).
\]
Thus, for all $f\in\mathfrak{B}$,
\begin{align*}|V_{\psi}f(x,\xi)| &\leq C \|M_{\xi} T_{x}\psi\|'_{p}
\\
&
\leq C \|\psi\|'_{p} \sum_{j\leq p}\binom{p}{j} |2\pi\xi|^{|p-j|}(1+|x|)^{p}\\
&
= C_{p} (1+|x|+|\xi|)^{2p} \|\psi\|'_{p},
\end{align*}
which gives that $C_{\psi}: \mathcal{S}'(\mathbb{R}^{d})\to \mathcal{S}'(\mathbb{Z}^{2d})$ is bounded upon setting again $x=\alpha k$ and $\xi=\beta n$.
 \end{proof}
 
We now prove the following continuity result for the Gabor synthesis operator.
\begin{proposition}
\label{proposition synthesis op}
The Gabor synthesis operator $(\ref{Gabor synthesis op})$ is continuous in the ensuing four cases:
\begin{equation*}
D_{\psi}: \mathcal{S}_{\exp,\textnormal{pol}}(\mathbb{Z}^{2d}) \to\mathcal{K}_{1}(\mathbb{R}^{d}) , \quad  \ D_{\psi}: \mathcal{S}'_{\exp,\textnormal{pol}}(\mathbb{Z}^{2d})\to\mathcal{K}_1'(\mathbb{R}^{d}) \quad \quad (\psi\in \mathcal{K}_{1}(\mathbb{R}^{d})),
\end{equation*}
\begin{equation*}
D_{\psi}: \mathcal{S}(\mathbb{Z}^{2d})\to \mathcal{S}(\mathbb{R}^{d}), \quad \mbox{ and } \ D_{\psi}: \mathcal{S}'(\mathbb{Z}^{2d}) \to \mathcal{S}'(\mathbb{R}^{d}) \quad \quad  (\psi\in \mathcal{S}(\mathbb{R}^{d})).
\end{equation*}
Furthermore, the series in $(\ref{Gabor synthesis op})$ is unconditionally convergent in the corresponding space.
\end{proposition}
\begin{proof}
Consider the reflection in the second coordinate mapping $A_2: \{c_{k,n}\}\mapsto \{c_{k,-n}\}$. Since $D_{\psi}A_2:\mathcal{S}'_{\exp,\textnormal{pol}}(\mathbb{Z}^{2d})\to\mathcal{K}_1'(\mathbb{R}^{d})$ and $\ D_{\psi}A_2: \mathcal{S}'(\mathbb{Z}^{2d}) \to \mathcal{S}'(\mathbb{R}^{d})$ are the transposes of $C_{\psi}: \mathcal{K}_{1}(\mathbb{R}^{d})\to \mathcal{S}_{\exp,\textnormal{pol}}(\mathbb{Z}^{2d})$ and $C_{\psi}: \mathcal{S}(\mathbb{R}^{d})\to \mathcal{S}(\mathbb{Z}^{2d})$, respectively, the continuity of $D_{\psi}:\mathcal{S}'_{\exp,\textnormal{pol}}(\mathbb{Z}^{2d})\to\mathcal{K}_1'(\mathbb{R}^{d})$ and $\ D_{\psi}: \mathcal{S}'(\mathbb{Z}^{2d}) \to \mathcal{S}'(\mathbb{R}^{d})$ follows from Lemma \ref{lemma coeff op}. By the closed graph theorem, the continuity of $D_{\psi}: \mathcal{S}_{\exp,\textnormal{pol}}(\mathbb{Z}^{2d}) \to\mathcal{K}_{1}(\mathbb{R}^{d}$) and $D_{\psi}: \mathcal{S}(\mathbb{Z}^{2d}) \to\mathcal{S}(\mathbb{R}^{d})$ would follow from $D_{\psi}( \mathcal{S}_{\exp,\textnormal{pol}}(\mathbb{Z}^{2d})) \subseteq \mathcal{K}_{1}(\mathbb{R}^{d})$ and $D_{\psi}( \mathcal{S}(\mathbb{Z}^{2d})) \subseteq \mathcal{S}(\mathbb{R}^{d})$ and the latter two continuity cases. We only show that $D_{\psi}( \mathcal{S}_{\exp,\textnormal{pol}}(\mathbb{Z}^{2d})) \subseteq \mathcal{K}_{1}(\mathbb{R}^{d})$, as the proof of  $D_{\psi}( \mathcal{S}(\mathbb{Z}^{2d})) \subseteq \mathcal{S}(\mathbb{R}^{d})$ is similar and we therefore omit it. Let $\{c_{k,n}\}\in  \mathcal{S}_{\exp,\textnormal{pol}}(\mathbb{Z}^{2d})$, we need to show that
\begin{equation}
\label{eqp1}
\sum_{(k,n)\in \mathbb{Z}^{2d}} c_{k,n} M_{\beta n}T_{\alpha k} \psi
\end{equation}
converges in $\mathcal{K}_{1}(\mathbb{R}^{d})$. Since we already know that (\ref{eqp1}) converges in $L^{2}(\mathbb{R}^{d})$ and, clearly, the derivatives of (\ref{eqp1}) are uniformly convergent due to the fast decay of the coefficients, it is enough to prove that for each $l\in \mathbb{N}_0$
\begin{equation}
\label{limit eq}\lim_{_{\substack{M\to\infty \\ N\to\infty}}} \sup_{x\in \mathbb{R}^{d}, \,|s|\leq l}e^{l|x|}\left|\sum_{|k|> M}\sum_{|n|> N }c_{k,n}(M_{\beta n}T_{\alpha
k}\psi)^{(s)}\right|=0.
\end{equation}
We then have
\begin{align*}
&\sup_{x\in \mathbb{R}^{d}, \,|s|\leq l}e^{l|x|}\left|\sum_{|k|> M}\sum_{|n|> N }c_{k,n}(M_{\beta n}T_{\alpha
k}\psi)^{(s)}\right|
\\
&
\
\leq (2\pi\beta+1)^{l}\sup_{x\in \mathbb{R}^{d}, \, |s|\leq
l}e^{l|x|}\sum_{|k|> M}\sum_{|n|> N }
\sum_{j\leq s}\binom{s}{j}|n|^{|s-j|}|c_{k,n}| |\psi^{(j)}(x-\alpha k)|
\\
&
\leq (2\pi\beta+1)^{l}\|\psi\|_{l}\sup_{x\in \mathbb{R}^{d}, \, |s|\leq l}e^{l|x|}\sum_{|k|> M}\sum_{|n|>N
}\left(1+|n|\right)^{|s|}\, |c_{k,n}|e^{-l|x-\alpha k|}
\\
&\leq (2\pi\beta+1)^{l}\|\psi\|_{l}\sum_{|k|> M}\sum_{|n|> N }|c_{k,n}|(1+|
n|)^{l}e^{l|\alpha k|}
=O_{p}((1+N)^{-p}e^{-pM}), \quad \forall p>0.
\end{align*}
This actually proves that (\ref{eqp1}) is unconditionally convergent in $\mathcal{K}_{1}(\mathbb{R}^{d})$. The unconditionally convergence in the case of $\mathcal{S}(\mathbb{R}^{d})$ can be established in a similar fashion.
It remains to show that (\ref{eqp1}) is also unconditionally convergent in the distribution spaces. We only treat the case of $\mathcal{K}_{1}'(\mathbb{R}^{d})$, as the case of convergence in $\mathcal{S}'(\mathbb{R}^{d})$ is completely analogous. The strong neighborhoods of the origin in $\mathcal{K}_{1}'(\mathbb{R}^{d})$ are given by polars of bounded subsets $\mathfrak{B}\subset \mathcal{K}_{1}(\mathbb{R}^{d})$. Let $\{c_{k,n}\}\in\mathcal{S}'_{\exp,\textnormal{pol}}(\mathbb{Z}^{2d})$ and let $\mathfrak{B}\subset \mathcal{K}_{1}(\mathbb{R}^{d})$ be bounded. Find $l$ such that 
$$
C=\sup_{(k,n)\in\mathbb{Z}^{2d}}\left|c_{k,n}\right|e^{-l|k|}(1+|n|)^{-l}<\infty.
$$ 
We have
\begin{align*}
\left|\sup_{\varphi\in\mathfrak{B}}\left\langle \sum_{|k|> M}\sum_{|n|> N }c_{k,n}M_{\beta n}T_{\alpha
k}\psi,\varphi \right\rangle\right| &\leq  C\sum_{|k|> M}\sum_{|n|> N }e^{l|k|}(1+|n|)^{l}\sup_{\varphi\in\mathfrak{B}}|V_{\psi}\overline{\varphi}(\alpha k,\beta n)|
\\
&
=O_{p,\mathfrak{B}}((1+N)^{-p}e^{-pM}), \quad \forall p>0,
\end{align*}
where in the last estimate we have used Lemma \ref{lemma coeff op} to conclude that $C_{\psi}(\overline{\mathfrak{B}})$ is a bounded subset of $\mathcal{S}_{\exp,\textnormal{pol}}(\mathbb{Z}^{2d})$. This completes the proof.
\end{proof}

We now specialize our results to Gabor frames. For it, we first need to discuss the regularity of the (canonical) dual window (cf. Subsection \ref{Gabor frame}).  A deep result by Janssen \cite{Janssen} states that if  (\ref{GF}) is a Gabor frame for $L^{2}(\mathbb{R}^{d})$ with window $\psi\in\mathcal{S}(\mathbb{R}^{d})$, then the dual window $\gamma=S^{-1}\psi\in \mathcal{S}(\mathbb{R}^{d})$ (see also \cite[Cor.~13.5.4, p.~296]{gr01}). We have to establish an analogous regularity result for windows $\psi\in\mathcal{K}_1(\mathbb{R}^{d})$. We shall do so with the aid of results by B\"{o}lcskei and Janssen \cite{b-j} on exponential decay of dual windows and Chung-Kim-Lee characterization of the Hasumi-Silva space \cite{c-k-l-1997}.

\begin{lemma}
\label{lemma regularity dual window}
Let $\psi\in\mathcal{K}_1(\mathbb{R}^{d})$. If $G(\psi, \alpha, \beta)$ is a Gabor frame for $L^{2}(\mathbb{R}^{d})$, then the dual window $\gamma=S^{-1}\psi\in \mathcal{K}_1(\mathbb{R}^{d})$.
\end{lemma}
\begin{proof} We already know that $\gamma\in\mathcal{S}(\mathbb{R}^{d})$ and $\alpha \beta<1$; in particular
\begin{equation}
\label{eqd1} \sup_{x\in\mathbb{R}^{d}} |\gamma^{(p)}(x)|<\infty, \quad \forall p\in\mathbb{N}_{0}.
\end{equation}
One also has \cite{b-j}  \begin{equation}
\label{eqd2} \sup_{x\in\mathbb{R}^{d}} |\gamma(x)|e^{p|x|}<\infty, \quad \forall p\in\mathbb{N}_{0}.
\end{equation}
Actually B\"{o}lcskei and Janssen \cite{b-j} stated this in weaker terms in \cite[Thm.~1]{b-j} and only for one-dimensional windows, but their argument can readily be adapted to  show that (\ref{eqd2}) holds in the multidimensional case as well.

The inequalities (\ref{eqd1}) and (\ref{eqd2}) turn out to characterize the space $\mathcal{K}_{1}(\mathbb{R}^{d})$, as shown in \cite[Sect. 2]{c-k-l-1997} (where the notation $H(\mathbb{R}^{d})=\mathcal{K}_1(\mathbb{R}^{d})$ is employed). So $\gamma\in \mathcal{K}_1(\mathbb{R}^{d})$.
\end{proof}

The regularity of the dual window allows us to consider Gabor frame expansions of distributions and test functions.

\begin{corollary} \label{coGF}
Let $G(\psi,\alpha,\beta)$ be a Gabor frame with dual window $\gamma$.
\begin{itemize}
\item[$(i)$] If $ \psi\in \mathcal{K}_1(\mathbb{R}^{d})$, then for any $\varphi\in \mathcal{K}_1(\mathbb{R}^{d})$ and $f\in \mathcal{K}'_1(\mathbb{R}^{d})$ we have  the expansions
\begin{equation}
\label{eqGFexp1}
\varphi=\sum_{k,n\in
\mathbb{Z}^{d}}c_{k,n}^{\psi}(\varphi)M_{\beta n}T_{\alpha k}\gamma
\end{equation}
and
\begin{equation}
\label{eqGFexp2}
f=\sum_{k,n\in
\mathbb{Z}^{d}}c_{k,n}^{\psi}(f)M_{\beta n}T_{\alpha k}\gamma
\end{equation}
with unconditional convergence in $\mathcal{K}_1(\mathbb{R}^{d})$ and $\mathcal{K}'_1(\mathbb{R}^{d})$, respectively. Furthermore,
\begin{equation}
\label{oweq3'}\left\langle f,\varphi\right\rangle=\sum_{k,n\in \mathbb{Z}^{d}}c_{k,n}^{\psi}(f)c_{k,-n}^{\bar\gamma}(\varphi)=
 \sum_{k,n\in \mathbb{Z}^{d}}c_{k,-n}^{\bar\gamma}(f)c_{k,n}^{\psi}(\varphi).
  \end{equation}
  \item [$(ii)$] If $ \psi\in \mathcal{S}(\mathbb{R}^{d})$, then $(\ref{eqGFexp1})$, $(\ref{eqGFexp2})$, and $(\ref{oweq3'})$ hold for any $\varphi\in \mathcal{S}(\mathbb{R}^{d})$ and $f\in \mathcal{S}'(\mathbb{R}^{d})$ (with unconditional convergence in $\mathcal{S}(\mathbb{R}^{d})$ and $\mathcal{S}'(\mathbb{R}^{d})$, respectively).
\end{itemize}
\end{corollary}
\begin{proof} We have that $\operatorname*{id}_{L^{2}(\mathbb{R}^{d})}= D_{\gamma} C_{\psi}$, the unconditional convergence of (\ref{eqGFexp1}) on the test function spaces then follows from Lemma \ref{lemma coeff op} and Proposition \ref{proposition synthesis op}. Since $\operatorname*{id}= D_{\gamma} C_{\psi} $ on dense subspaces of $\mathcal{S}'(\mathbb{R}^{d})$ and $\mathcal{K}'_{1}(\mathbb{R}^{d})$, (\ref{eqGFexp2}) for distributions is also a consequence of Lemma \ref{lemma coeff op} and Proposition \ref{proposition synthesis op}. The relation (\ref{oweq3'}) is now trivial.
\end{proof}

Note that Corollary \ref{coGF} could also be obtained from results on modulation spaces (cf. \cite[Cor. 12.2.6]{gr01}) in combination with Lemma \ref{lemma regularity dual window}.

Summing up, we arrive at the main result of this section, a topological characterization of $\mathcal{K}'_{1}(\mathbb{R}^{d})$ and $\mathcal{S}'(\mathbb{R}^{d})$ via the Gabor frame coefficient operator.

\begin{theorem}
\label{mainthtopologyGF} Let $G(\psi,\alpha,\beta)$ be a Gabor frame.
\begin{itemize}
\item[$(i)$] If $ \psi\in \mathcal{K}_1(\mathbb{R}^{d})$, then the Gabor coefficient operators $(\ref{coeff op mappings 1})$ are isomorphisms of topological vector spaces (onto their images $C_{\psi}(\mathcal{K}_1(\mathbb{R}^{d}))$ and $C_{\psi}(\mathcal{K}'_1(\mathbb{R}^{d}))$).
  \item [$(ii)$] If $ \psi\in \mathcal{S}(\mathbb{R}^{d})$,  then the Gabor coefficient operators $(\ref{coeff op mappings 2})$ are isomorphisms of topological vector spaces (onto their images $C_{\psi}(\mathcal{S}(\mathbb{R}^{d}))$ and $C_{\psi}(\mathcal{S}'(\mathbb{R}^{d}))$).
\end{itemize}
\end{theorem}
\begin{proof} That the involved mappings are injective follows from Corollary \ref{coGF}. The inverses are continuous because of Proposition \ref{proposition synthesis op} and Corollary \ref{coGF}.
\end{proof}


Based on Theorem \ref{mainthtopologyGF} we can now easily obtain our desired Gabor analysis characterizations of bounded sets and convergence in the distribution spaces $\mathcal{K}_{1}'(\mathbb{R}^{d})$ and $\mathcal{S}'(\mathbb{R}^{d})$.

\begin{corollary} \label{corollaryGF1} Let $G(\psi,\alpha,\beta)$ be a Gabor frame.
\begin{itemize}
\item[$(i)$] Suppose that $ \psi\in \mathcal{K}_1(\mathbb{R}^{d})$. Then, a subset
$\mathfrak{B}\subset{\mathcal K}'_1(\mathbb {R}^{d})$ is bounded in ${\mathcal K}'_1(\mathbb{R}^{d})$ if and only if there is $\tau>0$ such that
 \begin{equation}\label{equation111}
 \sup_{ f\in\mathfrak{B}}\sup_{ (k,n)\in\mathbb{Z}^{2d}}\frac{|c_{k,n}^{\psi}(f)|}{e^{\tau|k|}(1+|n|)^{\tau}}<\infty.
 \end{equation}
\item [$(ii)$] Suppose $ \psi\in \mathcal{S}(\mathbb{R}^{d})$. Then, a subset
$\mathfrak{B}\subset{\mathcal S}'(\mathbb{R}^{d})$ is bounded in ${\mathcal S}'(\mathbb {R}^{d})$ if and only if there is $\tau>0$ such that
 \begin{equation}\label{equation11}
 \sup_{ f\in\mathfrak{B}}\sup_{ (k,n)\in\mathbb{Z}^{2d}}\frac{|c_{k,n}^{\psi}(f)|}{(1+|k|+|n|)^{\tau}}<\infty.
 \end{equation}

\end{itemize}
\end{corollary}

\begin{proof} By Theorem \ref{mainthtopologyGF}, we have that $\mathfrak{B}$ is bounded if and only if $C_{\psi}(\mathfrak{B})$ is bounded. The estimates (\ref{equation111}) and (\ref{equation11}) are equivalent to the boundedness of $C_{\psi}(\mathfrak{B})$ in $\mathcal{S}'_{\exp,\textnormal{pol}}(\mathbb{Z}^{2d})$ and $\mathcal{S}'(\mathbb{Z}^{2d})$, respectively, because, as Silva spaces, they are the regular inductive limits of Banach spaces.
\end{proof}

Likewise, convergence of nets of distributions can also be characterized in terms of Gabor frame coefficients.  We shall make use of the following corollary in the next section.

\begin{corollary} \label{corollaryGF2} Let $G(\psi,\alpha,\beta)$ be a Gabor frame with dual window $\gamma$.
\begin{itemize}
\item[$(i)$] Suppose that $ \psi\in \mathcal{K}_1(\mathbb{R}^{d})$. Then, a net $\{f_{\lambda}\}_{\lambda\in \mathbb{R}_{+}}$ of distributions of exponential type converges in ${\mathcal K}'_1(\mathbb{R}^{d})$ as $\lambda\to\infty$ if and only if
\begin{equation}
\label{equation112}
\lim_{\lambda\to\infty}c^{\psi}_{k,n}(f_{\lambda})=a_{k,n}<\infty\ , \quad \mbox{for each }(k,n)\in\mathbb{Z}^{2d},
\end{equation}
and there are $\tau>0$ and $\lambda_0$ such that
 \begin{equation}\label{equation113}
 \sup_{ \lambda\geq\lambda_0}\sup_{ (k,n)\in\mathbb{Z}^{2d}}\frac{|c_{k,n}^{\psi}(f_{\lambda})|}{e^{\tau|k|}(1+|n|)^{\tau}}<\infty.
 \end{equation}
\item [$(ii)$] Suppose $ \psi\in \mathcal{S}(\mathbb{R}^{d})$.
Then, a net $\{f_{\lambda}\}_{\lambda\in \mathbb{R}_{+}}$ of tempered distributions converges in ${\mathcal S}'(\mathbb{R}^{d})$ as $\lambda\to\infty$ if and only if it satisfies $(\ref{equation112})$ and there are $\tau>0$ and $\lambda_0$ such that
 \begin{equation}\label{equation114}
 \sup_{ \lambda\geq\lambda_0}\sup_{ (k,n)\in\mathbb{Z}^{2d}}\frac{|c_{k,n}^{\psi}(f_{\lambda})|}{(1+|k|+|n|)^{\tau}}<\infty.
 \end{equation}

\end{itemize}
In such a case the limit functional, $\lim_{\lambda\to\infty}f_{\lambda}=g$, is given by
$
\displaystyle g=\sum_{k,n\in
\mathbb{Z}^{d}}a_{k,n}M_{\beta n}T_{\alpha k}\gamma .
$
\end{corollary}

\begin{proof} The direct part follows from the weak convergence and Theorem \ref{mainthtopologyGF}. Conversely, assume (\ref{equation112}) and (\ref{equation113}) (resp. (\ref{equation114})). Corollary \ref{corollaryGF1} and the Banach-Steinhaus theorem yield that the net forms an equicontinuous set. The assumption (\ref{equation112}) gives convergence of net on the linear span of $G(\psi, \alpha,\beta)$, which is dense by interchanging the roles of $\psi$ and $\gamma$ in Theorem \ref{mainthtopologyGF}. Therefore, the net converges in the topology of uniform convergence over compact subsets of $\mathcal{K}_1(\mathbb{R}^{d})$ ($\mathcal{S}(\mathbb{R}^{d})$, resp.). The rest follows from the Montel property of these spaces.
\end{proof}

\begin{remark} A question that remains open is whether analogs to the results of this section can be obtained for the space of Lizorkin distributions $\mathcal{S}'_{0}(\mathbb{R})$ (the quotient of $\mathcal{S}'(\mathbb{R})$ by the space of polynomials \cite{holschneider}) in terms of \emph{wavelet frames}. Two of the authors have provided such results when using orthogonal wavelets \cite{Saneva1}. Some difficulties of working with wavelet frames have been pointed out in \cite[Remark 3.11]{Saneva1}. \end{remark}

\section{Asymptotic behavior of distributions. Tauberian theorems}\label{Section Asymptotic behavior of distributions}

In this section we apply our Gabor frame characterization of convergence (Corollary \ref{corollaryGF2}) to characterize asymptotic properties of Schwartz distributions. Our results can be interpreted as Tauberian theorems for the STFT.

We are interested in the so-called $S$-asymptotic behavior of distributions \cite[Chap.~1]{PSV}. The natural framework for this notion is the space of distributions of exponential type. The idea of the $S$-asymptotics is to study the asymptotic properties of the translates $T_{-h}f$ with respect to a measurable comparison function  $c:\mathbb{R}^{d}\to (0,\infty)$. We say that $f\in\mathcal{D}'(\mathbb{R}^{d})$ has $S$-asymptotic behavior ($S$-asymptotics) with respect to $c$ if there is $g\in\mathcal{D}'(\mathbb{R}^{d})$ such that
\begin{equation}
\label{Seq1}\lim_{|h|\to\infty} \frac{1}{c(h)}T_{-h}f= g \  \  \ \mbox{in }  \mathcal{D}'(\mathbb{R}^{d}).
\end{equation}
The distribution $g$ is not arbitrary; in fact, one can show \cite[Sect. 1.2]{PSV} that the relation (\ref{Seq1}) forces it to have the form $g(x)=C e^{b\cdot x}$, for some $C\in\mathbb{R}$ and $b\in \mathbb{R}^{d}$. Furthermore, if $C\neq 0$, one can also prove \cite[Sect. 1.2]{PSV} that 
\begin{equation}
\label{Seq2}\lim_{|h|\to\infty} \frac{c(x+h)}{c(h)}= e^{b \cdot x},  \  \  \ \mbox{uniformly for }x \mbox{ in compact subsets of } \mathbb{R}^{d}.  
\end{equation}

We shall assume that (\ref{Seq2}) is always satisfied. One can then show that necessarily $f\in\mathcal{K}'_{1}(\mathbb{R}^{d})$ and that the limit (\ref{Seq1}) actually holds in $\mathcal{K}'_1(\mathbb{R}^{d})$. (This follows from the structural theorem \cite[Thm.~1.9, p.~44]{PSV} and (\ref{weightestimate}) below.) It is also easy to see that (\ref{Seq2}) implies that $c$ is locally bounded for large arguments. Since only the terminal behavior of $c$ matters for the $S$-asymptotics (\ref{Seq1}), we can also suppose from now on that $c\in L^{\infty}_{loc}(\mathbb{R}^{d})$ by simply modifying it away from a neighborhood of $\infty$. Under these circumstances, we have the estimate \cite[Lemma 6.1]{KPSV}
\begin{equation}
\label{weightestimate}
\frac{c(x+h)}{c(h)}\leq Ae^{r|x|},  \quad x,h\in\mathbb{R}^{d}, \quad \mbox{for some }r,A>0.
\end{equation}

We will use the more suggestive notation
\begin{equation}
\label{Seq3}f(x+h)\sim c(h)g(x)  \  \  \ \mbox{ in } \mathcal{K}'_{1}(\mathbb{R}^{d})\  \mbox{ as }|h|\to\infty
\end{equation}
for denoting (\ref{Seq1}). The next theorem shows that it is possible to improve \cite[Thm. ~6.2]{KPSV} by discretizing the frequency variable in the STFT if one employs a window $\psi$ that generates a Gabor frame.

\begin{theorem} \label{tttGF1} Let $G(\psi,\alpha,\beta)$ be a Gabor frame with window $\psi\in \mathcal{K}_{1}(\mathbb{R}^{d})$ and let $c\in L^{\infty}_{loc}(\mathbb{R}^{d})$ satisfy $(\ref{Seq2})$. Then, $f\in\mathcal{K}_{1}'(\mathbb{R}^{d})$ has the $S$-asymptotic behavior $(\ref{Seq3})$ if and only if
\begin{equation}\label{limit2} \lim_{|x|\to \infty }e^{2\pi i \beta n \cdot x}\frac{V_{\psi}f(x,\beta n)}{c(x)}=:a_n\in\mathbb{C}\quad \mbox{exists for every } n\in\mathbb{Z}^{d}\end{equation}
and there is $\tau\in\mathbb{R}$ such that
\begin{equation}\label{Taubeq}
\sup_{(x,n)\in\mathbb {R}^{d}\times \mathbb{Z}^{d}}\frac{|V_\psi f(x,\beta n)|}{c(x)(1+|n|)^{\tau}}<\infty.
\end{equation}
If this is the case, the limit function involved in the S-asymptotic behavior $(\ref{Seq3})$ has the form $g(x)=C e^{b\cdot x}$ where the constants satisfy the equations
\begin{equation}
\label{eqconstants}
a_n=C \overline{\widehat{\psi}\left(-\beta n+ib/(2\pi)\right)}, \quad n\in\mathbb{Z}^{d}.
\end{equation}
\end{theorem}
\begin{proof} We are going to apply Corollary
\ref{corollaryGF2}. For it, consider the net
\begin{equation}
\label{eq net translation}
f_{h}=\frac{1}{c(h)}T_{-h}f,
\end{equation}
for $h\in\mathbb{R}^{d}$. We have that the Gabor frame coefficients of the net are given by
 $$
 c^{\psi}_{k,n}(f_h)=e^{2\pi i \beta n \cdot h}\frac{V_{\psi}(\alpha k+h,\beta n)}{c(h)}.
 $$
 Assume that $f\in\mathcal{K}_{1}'(\mathbb{R}^{d})$ satisfies (\ref{Seq3}). The limit (\ref{equation112}) yields in this case (\ref{limit2}) with
 $a_n=c^{\psi}_{0,n}(g)=\int_{\mathbb{R}^{d}}g(t)\overline{\psi(t)}e^{-2\pi i \beta n\cdot t}dt$. Since $g$ must have the form $g(x)=Ce^{b\cdot x}$, this also shows that $C$ and $b$ must satisfy the equations (\ref{eqconstants}), while (\ref{equation113}) directly leads to (\ref{Taubeq}) by taking $k=0$. Conversely, suppose that (\ref{limit2}) and (\ref{Taubeq}) hold. In view of (\ref{Seq2}), we have that
 \begin{align*}
\lim_{|h|\to\infty}c^{\psi}_{k,n}(f_h)&=e^{-2\pi i \alpha \beta k\cdot n}\lim_{|x|\to\infty}e^{2\pi i \beta n \cdot x}\frac{V_{\psi}(x,\beta n)}{c(x-k\alpha)}=e^{-2\pi i \alpha \beta k\cdot n}a_n\lim_{|x|\to\infty}\frac{c(x)}{c(x-k\alpha)}
\\
&
=e^{\alpha k\cdot b-2\pi i \alpha \beta k\cdot n}a_n, \end{align*}
which shows that the net $\{f_{h}\}_{h\in\mathbb{R}^{d}}$ satisfies the hypothesis (\ref{equation112}) from Corollary \ref{corollaryGF2}. On the other hand, the estimates (\ref{weightestimate}) and (\ref{Taubeq}) imply that
$$
\sup_{ h\in\mathbb{R}^{d}} \sup_{ (k,n)\in\mathbb{Z}^{2d}}\frac{|c_{k,n}^{\psi}(f_{h})|}{e^{\tau'|k|}(1+|n|)^{\tau'}}<\infty,
$$
with $\tau'=\max\{\tau,r\alpha\}$. Corollary \ref{corollaryGF2} then yields the result.
\end{proof}

In the rest of this section we focus on one-dimensional distributions. In this case it makes sense to consider the $S$-asymptotic behavior for $h\to\infty$, that is, 
\begin{equation}
\label{Seq3.1}f(x+h)\sim c(h)g(x)  \  \  \ \mbox{ in } \mathcal{K}'_{1}(\mathbb{R})\  \mbox{ as }h\to\infty,
\end{equation}
where $c:\mathbb{R}\to(0,\infty)$. A version of Theorem \ref{tttGF1} then applies to characterize (\ref{Seq3.1}); indeed, one just needs to replace $|x|\to\infty$ by $x\to\infty$ in (\ref{limit2}) to obtain the desired characterization. Since $g(x)=C e^{bx}$, one can write \cite{PSV} $c(h)=e^{bh}L(e^{h})$, where $L$ is a Karamata slowly varying function (at infinity), i.e., one that satisfies
\begin{equation*}
\lim_{\lambda\rightarrow\infty}\frac{L(a\lambda)}{L(\lambda)}=1\ \ \text{for each}\ a>0.
\end{equation*}
As an application, we deduce from Theorem \ref{tttGF1} the ensuing Tauberian theorem for the STFT of non-decreasing functions.

\begin{theorem}\label{ttGF2} Let $f$ be a positive non-decreasing function on $[0,\infty)$, $let$ $L$ be a slowly varying function, and
let $G(\psi,\alpha,\beta)$ be a Gabor frame with nonnegative window $\psi\in \mathcal{K}_{1}(\mathbb{R})$.
 Suppose that the limits
\begin{equation}
\label{eq01}\lim_{x\to\infty}\frac{e^{2\pi i n x}}{e^{b x}L(e^{x})} \int _{0}^{\infty} f(t)\psi(t-x)e^{-2\pi i nt} \ dt=: a_n
\end{equation} exist for all $n\in\mathbb{Z}$. Then,
\begin{equation}
\label{eq02}
\lim_{x\to\infty} \frac{f(x)}{e^{b x}L(e^{x})}= \frac{a_0}{\int_{-\infty}^{\infty}\psi(t)e^{b t} dt}\ .
\end{equation}
\end{theorem}
\begin{proof}
First notice that we must have $b\geq 0$. We show $f(t)=O(c(t))$, where $c(t)=e^{bt}L(e^{|t|})$. Set $C_{1}=\int_{0}^{\infty}\psi(t)dt<\infty$. Since $f$ is non-decreasing, we have
$$
f(x)\leq \frac{1}{C_{1}}\int_{0}^{\infty}f(t+x)\psi(t)dt\leq \frac{1}{C_1}\int_{0}^{\infty}f(t)\psi(t-x)dt=O(c(x)),
$$
because of (\ref{eq01}) with $n=0$. Making use of (\ref{weightestimate}), we conclude that
$$
|V_{\psi}f(x,\beta n)|\leq C_{2}\int_{0}^{\infty}c(t)\psi(t-x)dt\leq C_{3}c(x)\int_{-\infty}^{\infty}e^{r|t|}\psi(t)dt< O(c(x)).
$$
Theorem \ref{tttGF1} implies that $f(x+h)\sim c(h) C e^{bx}$ in $\mathcal{K}'_{1}(\mathbb{R})$ as $h\to\infty$, where $C$ is given by the limit (\ref{eq02}). The rest of the proof goes along the same lines as that of \cite[Thm.~6.5]{KPSV}. Choose a non-negative test function $\varphi\in\mathcal{D}(\mathbb{R})$ such that $\operatorname*{supp}\varphi\subseteq(0,\varepsilon)$ and $\int_{0}^{\varepsilon}\varphi(t)\mathrm{d}t=1$. Using the fact that $f$ is non-decreasing on $(0,\infty)$ and (\ref{Seq3}), we obtain
\begin{equation*}
\limsup_{h\to\infty}\frac{f(h)}{c(h)}\leq\lim_{h\to\infty}\frac{1}{c(h)}\int_{0}^{\varepsilon}f(t+h)\varphi(t)dt
=C \int_{0}^{\varepsilon}e^{b t}\varphi(t)dt \leq C e^{b \varepsilon};
\end{equation*}
taking $\varepsilon\to0^{+}$, we have shown that $\limsup_{h\to\infty}f(h)/c(h)\leq C$. Similarly, one obtains $\liminf_{h\to\infty}f(h)/c(h)\geq C$.
\end{proof}

Theorem \ref{tttGF1} can be strengthened if additionally the window $\psi$ is a Wiener type kernel. In the next theorem we make use of the Pilipovi\'{c}-Stankovi\'{c} distributional version of the Wiener Tauberian theorem \cite{p-sWiener, PSV}.

\begin{theorem} \label{ttGF3} Let $G(\psi,\alpha,\beta)$ be a Gabor frame with window $\psi\in \mathcal{K}_{1}(\mathbb{R})$ and let $c(h)=e^{bh}L(e^{|h|})$, where $L$ is a locally bounded slowly varying function. Assume that
\begin{equation}
\label{eqWienercond}
\widehat{\psi}\left(\xi+\frac{ib}{2\pi}\right)=\int_{-\infty}^{\infty}\psi(t)e^{bt}e^{-2\pi i \xi t}dt\neq 0, \quad \mbox{for all }\xi\in\mathbb{R}.
\end{equation}
If
\begin{equation*}
 \lim_{x\to \infty }\frac{V_{\psi}f(x,0)}{e^{bx}L(e^{x})}= \lim_{x\to \infty }\frac{(f\ast \check{\psi})(x)}{e^{bx}L(e^{x})}= a_0\in\mathbb{C}\quad \mbox{exists}
 \end{equation*}
and the Tauberian condition $(\ref{Taubeq})$ holds for some $\tau$ and $c$ extended as $c(x)=e^{-\tau x}$ for $x\leq 0$, then $f$ has S-asymptotic behavior $(\ref{Seq3.1})$ with $g(x)=C e^{b x}$, where
\begin{equation*}
C= \frac{a_0}{\int_{-\infty}^{\infty}\psi(t)e^{b t} dt}\ .
\end{equation*}
\end{theorem}

\begin{proof} Inspection in the proof of Theorem \ref{tttGF1} shows that the Tauberian condition (\ref{Taubeq}) is equivalent to the fact that the net (\ref{eq net translation}) is bounded. The $S$-asymptotic behavior (\ref{Seq3}) now follows directly by applying \cite[Prop. 4.8, p. 206]{PSV}.
\end{proof}

\begin{example}[The Gaussian window] Let $\beta$ be any positive number. If $\psi(x)=e^{-\pi x^2}$ and we choose any $\alpha$ with the property $\alpha \beta<1$, then $G(\psi,\alpha,\beta)$ is a Gabor frame for $L^{2}(\mathbb{R})$. Since, $\hat{\psi}(\xi)=e^{-\pi \xi^{2}}$, the Wiener type condition (\ref{eqWienercond}) is satisfied for all $b\in \mathbb{R}$. In this special case, Theorem \ref{ttGF3} yields
$$f(x+h)\sim a_0e^{-\frac{b^{2}}{4\pi}}e^{bh}L(e^{h})e^{bx} \  \  \ \mbox{ in } \mathcal{K}'_{1}(\mathbb{R})\  \mbox{ as }h\to\infty
.$$
Furthermore, one can also derive a form of the Wiener Tauberian theorem for the Gaussian kernel from Theorem \ref{ttGF3}. \emph{Let $f$ be non-decreasing on $[0,\infty)$ and let $L$ be slowly varying. If}
$$\lim_{x\to\infty}\frac{1}{e^{b x}L(e^{x})} \int _{0}^{\infty} f(t)e^{-\pi(t-x)^{2}} \ dt=a_0
,$$
\emph{then,} $f(x)\sim a_0e^{-\frac{b^{2}}{4\pi}}L(e^{x})e^{bx}$ \emph{as} $x\to\infty$. The proof of this assertion is exactly the same as that of Theorem \ref{ttGF2}, but employing Theorem \ref{ttGF3} instead of Theorem \ref{tttGF1}.
\end{example}

\smallskip

We end this article with a remark.

\begin{remark}[The case of tempered distributions]  If  $c(h)\sim|h|^{\nu}L(|h|)$ where $L$ is slowly varying and (\ref{Seq1}) is satisfied, then $f\in\mathcal{S}'(\mathbb{R}^{d})$ and actually (\ref{Seq1}) holds in the (strong) topology of $\mathcal{S}'(\mathbb{R}^{d})$ (this follows from \cite[Thm.~1.9, p.~44]{PSV} and the remarks in \cite[p. 45]{PSV}). Hence, as a consequence of the second part of Corollary \ref{corollaryGF2}, in this case we may use Gabor windows $\psi\in\mathcal{S}(\mathbb{R}^{d})$ in Theorem \ref{tttGF1}, Theorem \ref{ttGF2}, and Theorem \ref{ttGF3}.  \end{remark}

\end{document}